\documentclass{amsart}
\usepackage{chngcntr}
\usepackage{apptools}
\usepackage{color}
\usepackage{amsthm,amssymb,verbatim}
\usepackage{graphicx}
\usepackage{enumerate}
\usepackage{chngcntr}
\usepackage{apptools}
\usepackage{color}
\usepackage{amsthm,amssymb,verbatim}
\usepackage{graphicx}
\usepackage{enumerate}

\newcommand{\Ff}{\mathcal F}

 \newcommand{\Cc}{\mathcal C}

 \newcommand{\Vv}{\mathcal{V}}

  \newcommand{\MM}{\mathcal{M}}
  
\newcommand{\Pp}{\mathcal{P}}

 \newcommand{\RR}{\mathbf{R}}  
 \newcommand{\BB}{\mathbf{B}}  

    \newcommand{\dist}{\operatorname{dist}}

 \newcommand{\eps}{\epsilon}
 \newcommand{\Tan}{\operatorname{Tan}}
 
 \newcommand{\Kk}{\mathcal{K}}
 \newcommand{\UU}{\mathcal{U}}

 \newcommand{\NN}{\mathbf{N}}

\newcommand{\ee}{\mathbf e}

\usepackage{amsthm}

\input epsf
\def\begfig {
\begin{figure}
\small }
\def\endfig {
\normalsize
\end{figure}
}

    \newtheorem{theorem}    {Theorem}   
    \newtheorem{lemma}      [theorem]       {Lemma}
    \newtheorem{corollary}  [theorem]     {Corollary}
    \newtheorem{proposition}       [theorem]       {Proposition}
    
    \newtheorem{claim}{Claim}
    \newtheorem*{theorem*}{Theorem}
    \theoremstyle{definition}
    \newtheorem{definition}  [theorem] {Definition}
     
    \theoremstyle{definition}
    \newtheorem{remark}   [theorem]       {Remark}

\title[Generic Transversality]{Generic Transversality of Minimal Submanifolds and Generic Regularity 
of Two-Dimensional Area-Minimizing Integral Currents}
\author{Brian White}
\thanks{The author was partially supported by NSF grant DMS-1711293}
\address{Department of Mathematics\\ Stanford University\\ Stanford, CA 94305}
\email{bcwhite@stanford.edu}
\date{January 15, 2019. Revised December 2, 2019}
\usepackage{hyperref}
\usepackage{enumerate}
\usepackage[alphabetic, msc-links, backrefs]{amsrefs}
\begin{document}

\maketitle

\begin{abstract}
Suppose that $N$ is a smooth manifold with a smooth Riemannian metric $g_0$, and
that $\Gamma$ is a smooth submanifold of $N$.
This paper proves that for a generic (in the sense of Baire category) smooth metric $g$ conformal to $g_0$,
if $F$ is any simple $g$-minimal immersion of a closed manifold into $N$, then
$F$ is transverse to $\Gamma$, and $F$ is self-transverse.

The paper also proves that for a generic ambient metric, every $2$-dimensional surface
(integral current or flat chain mod $2$) without boundary that minimizes area in its homology class
has support equal to a smoothly embedded minimal surface.
\end{abstract}

\section{Introduction}

In this paper, we prove:

\begin{theorem}\label{main-theorem}
Suppose that $N$ is a smooth manifold with a smooth Riemannian metric $g_0$, and
that $\Gamma$ is a smooth submanifold of $N$.
For a generic (in the sense of Baire category) smooth metric $g$ conformal to $g_0$,
if $F$ is any simple $g$-minimal immersion of a closed manifold into $N$, then
\begin{enumerate}[\upshape (1)]
\item   $F$ is transverse to $\Gamma$, and
\item  $F$ is self-transverse.
\end{enumerate}
\end{theorem}

An immersion $F:M\to N$ is called {\bf self-transverse}  provided the following holds: if  $U$ and $W$
are disjoint open sets in $M$ and if the restrictions $F|_U$ and $F|_W$ are embeddings, then $F(U)$ and $F(W)$ are
transverse.

An immersion $F:M\to N$ is called {\bf simple} if each connected component of $M$ contains
a point $q$ such that $F(q)$ and $F(M\setminus \{q\})$ are disjoint.   In case $F$ is $g$-minimal
for a smooth metric $g$, unique continuation implies that if $F$ is simple, $F(q)$ and $F(M\setminus q)$
are disjoint except for a closed, nowhere dense, measure-$0$ set of $q\in M$.
Thus a $g$-minimal immersion is simple if and only if the image has multiplicity $1$ almost everywhere.

Theorem~\ref{main-theorem} is false without the word ``simple", even in the case of $1$-dimensional minimal 
submanifolds (i.e., geodesics).   For there is a nonempty open set of metrics on $N$
for which there is a closed geodesic.  If we traverse the geodesic multiple times,
the result is a closed geodesic with non-transverse self-intersections.

We also prove a stronger version of Theorem~1:

\begin{theorem}\label{strong-main-theorem}
Theorem~1 remains true with ``strongly transverse'' and ``strongly-self-transverse"
in place of ``transverse" and ``self-transverse".
\end{theorem}

Strong transversality and strong self-transversality are defined in Section~\ref{strong-section}.
Theorem~\ref{more-intuitive-theorem} in Section~\ref{equivalence-section} gives a more geometrically intuitive
characterization of those terms.
The terminology is easiest to understand
when $M$ and $\Gamma$ are
hypersurfaces in $N$ (with $M$ immersed and $\Gamma$ embedded).
In that case,
\begin{enumerate}
\item $M$ is strongly self-transverse if for each point $p\in N$, the unit normals
  to the sheets of $M$ passing through $p$ are linearly independent.
\item $M$ is strongly transverse to $\Gamma$ if for each point $p\in \Gamma$,
  the unit normal to $\Gamma$ at $p$ and the unit normals to the sheets of $M$ passing through $p$
  are linearly independent.
\end{enumerate}
See Theorem~\ref{more-intuitive-theorem} in \S\ref{equivalence-section}.

Theorems~\ref{main-theorem} and~\ref{strong-main-theorem}
also hold for constant mean curvature immersions and more generally 
for prescribed mean curvature immersions.
See~\S\ref{prescribed-section}.

The proofs of Theorems~\ref{main-theorem} and~\ref{strong-main-theorem} are based on
the Bumpy Metrics Theorem (see~\S\ref{bumpy-section}) 
together with a very general theorem (Theorem~\ref{PDE-theorem}) 
about linear elliptic partial differential equations.
The flavor of the PDE Theorem is indicated by the following (which is equivalent to a special case
of that theorem):

\begin{theorem}
Let $M$ be a smooth, compact, connected Riemannian manifold with smooth, nonempty boundary.
Let $S$ be a finite subset of the interior of $M$ and let $f:S\to \RR$ be any function.  Then there is a harmonic
function $h$ on $M$ such that $h(x)=f(x)$ for each $x\in S$.
\end{theorem}

The PDE Theorem is a rather direct consequence of a theorem of Peter Lax.

The transversality results of this paper play a key role in Xin Zhou's proof~\cite{zhou-multiplicity} of 
 the Marques-Neves multiplicity-one conjecture.
 Indeed, this paper grew out of a conversation in which Professor Zhou
 explained to me how better knowledge of generic behavior of prescribed
 mean curvature surfaces could be very useful in min-max theory.
 
In~\S\ref{regularity-section}, we apply the results in the preceding sections to show that for a generic smooth
Riemannian metric on a manifold $N$, every $2$-dimensional locally area-minimizing integral
current without boundary has support equal to a smoothly embedded minimal
surface.  The same is true for flat chains mod $2$.  In particular, for generic ambient metrics, $2$-dimensional
varieties that minimize in their homology classes (integral or mod $2$)
are smoothly embedded minimal surfaces, possibly (in the integral case) with
multiplicity.

\section{The Bumpy Metrics Theorem}\label{bumpy-section}

\newcommand{\MMreg}{\MM_\textnormal{reg}}
\newcommand{\MMsing}{\MM_\textnormal{sing}}

Let $N$ be a smooth manifold with a smooth Riemannian metric $g_0$.

Two smooth immersions $F_i:M_i\to N$ of closed manifolds into $N$ are called
{\bf equivalent} if there is a smooth diffeomorphism $u:M_1\to M_2$ such that $F_2=F_1\circ u$.
If $F$ is a smooth immersion into $N$, we let $[F]$ denote its equivalence class.
If $F_i$ and $F$ are smooth immersions, we say that $[F_i]$ converges smoothly to $[F]$
if there are immersions $F_i'\in [F_i]$ such that $F_i'$ converges smoothly to $F$.

Let $\MM$ be the space of all pairs $(\gamma, [F])$ such that $\gamma\in C^\infty(N)$
and $F$ is a smooth, simple, $e^{\gamma}g_0$-minimal immersion of a closed manifold into $N$.

Define a projection $\Pi$ by
\begin{align*}
&\Pi:\MM \to C^\infty(N),  \\
&\Pi(\gamma,[F])=\gamma.
\end{align*}

Let $\MMreg$ be the union of open sets $U\subset \MM$ such that $\Pi$ maps
$U$ homeomorphically onto an open subset of $C^\infty(N)$.
It follows from the implicit function theorem that $(\gamma,[F])\in \MMreg$
if and only if $[F]$ has no nonzero Jacobi fields (for the metric $e^{\gamma}g_0$).

Let $\MMsing= \MM\setminus \MMreg$.

\begin{theorem}[Bumpy Metrics Theorem]\label{bumpy-theorem}
The set $\Pi(\MMsing)$ is a meager subset of $C^\infty(N)$.
\end{theorem}

For proof, see~\cite{white-bumpy}.

\begin{corollary}[Bumpy Metrics Corollary]\label{bumpy-corollary}
Suppose that $\Kk$ is a closed subset of $\MM$, 
or, more generally, a relatively closed subset of an open subset $\UU$
of $\MM$.
 Suppose also that every nonempty open subset of $\MMreg$
contains contains a point of $\MM\setminus \Kk$.
Then $\Pi(\Kk)$ is meager in $C^\infty(N)$.
\end{corollary}

\begin{proof}
  Because $\MM$ is second countable (see Remark~\ref{second-countable}),
   $\MMreg\cap\UU$ is a countable union of open sets $\UU_i$
such that $\Pi$ maps $\UU_i$ homeomorphically onto a open subset $\Pi(\UU_i)$ of $C^\infty(N)$.
Now
\begin{equation}\label{covered}
  \Pi(\Kk) \subset \Pi(\MMsing) \cup \left(\cup_i \Pi(\Kk\cap \UU_i)\right).
\end{equation}
Since $\Kk\cap \UU_i$ is a relatively closed subset of $\UU_i$,
 the hypothesis of the Corollary implies that $\Kk\cap \UU_i$
is nowhere dense in $\UU_i$.  Hence $\Pi(\Kk\cap \UU_i)$ is nowhere dense in $\Pi(U_i)$ and therefore is nowhere
dense in $C^\infty(N)$. Thus by~\eqref{covered} and the Bumpy Metrics Theorem,
 $\Pi(\Kk)$ is meager in $C^\infty(N)$.
\end{proof}

\begin{remark}\label{second-countable}
Second countability of $\MM$ may be proved as follows.
Up to smooth diffeomorphism, there are only countably many smooth, closed $m$-manifolds.
Let $M_1, M_2,\dots$ be an enumeration of them.  Give each $M_i$ a smooth Riemannian metric.
For each $i$, let $\Cc_i$ be a countable dense subset of $C^\infty(N)\times C^\infty(M_i,N)$.  
For $(\gamma,F)\in \Cc_i$, $j\in\NN$, and $k\in \NN$, let $U_{(\gamma',F'),j,k}$ be the set 
of $(\gamma,[F])$ in $\MM$ such that:
\begin{gather*}
\textnormal{domain}(F')=M_i, \\
\| F' - F \|_{C^j} < \frac1k, \\
\|\gamma' - \gamma\|_{C^j} < \frac1k,
\end{gather*}
(where the $C^j$ norms are with respect to the background metric $g_0$ on $N$.)
Then the $U_{(\gamma,F'),j,k}$ form a countable basis for topology of $\MM$.
(To make sense of $F'-F$, we isometrically embed $(N,g_0)$ in some Euclidean space.)
\end{remark}

\section{The Mean Curvature Operator}\label{mean-curvature-operator-section}

Let $N$ be a smooth $n$-dimensional manifold with a smooth Riemannian metric $g$.  
Let $M$ be a smooth, closed manifold,
and $F: M\to N$ be a smooth, $g$-minimal immersion.  Let $\Vv_F$ be the space of
all smooth normal vectorfields to $F$.  Thus $f\in \Vv_F$ if and only if $f$ 
is a smooth function that assigns to each $x\in M$ a vector $f(x)$ in $\Tan_{F(x)}N$
that is perpendicular to the image of $DF(x)$.

If $u\in \Vv_F$ is sufficiently small (in $C^1$ norm), then 
\begin{equation}\label{the-immersion}
   x\in M\mapsto F(x) + u(x)
\end{equation}
is also an immersion. 

(The expression~\eqref{the-immersion} makes sense if $N$ is $\RR^n$.
In a general ambient manifold $N$, the right hand side of~\eqref{the-immersion} 
should be replaced by the image of $u(x)$ under the exponential map.)

 If $\gamma\in C^\infty(N)$, the immersion~\eqref{the-immersion}
is minimal with respect to the Riemannian metric $e^\gamma g$
if and only if $u$ satisfies the relevant Euler-Lagrange system:
\[
    H(\gamma,u) = 0.
\]
Here $H(\gamma,\cdot):\Vv_F\to \Vv_F$ is a second-order quasilinear elliptic operator.

Let $G=D_1H(0,0)$ and $J=D_2H(0,0)$.
Then $J$ is the Jacobi operator, a second-order, self-adjoint, linear elliptic operator; it is the 
sum of the Laplace operator and a zero-order operator.
The Jacobi operator
 reflects how the mean curvature changes (to first order) as we move the surface while keeping the metric fixed.

The operator $G:C^\infty(N)\to \Vv_F$ is a linear differential operator that reflects how
 the mean curvature changes (to first order) as we vary the metric while keeping the surface fixed.
In fact, one easily calculates that 
\begin{equation}\label{G-formula}
    G(\gamma)(p) =  -\frac{m}2 ((\nabla u)(F(p)))^\perp,
\end{equation}
where $m=\dim(M)$ and the $\perp$ indicates the projection onto the orthogonal complement of $\Tan(F,p)$.

In general, the map $G:\Vv\to C^\infty(N)$ need not be surjective.
For example, if $p_1$ and $p_2$ are distinct points in $M$ with
$F(p_1)=F(p_2)$ and $\Tan(F,p_1)=\Tan(F,p_2)$, then for
any $\gamma$,  we have $G(\gamma)(p_1)=G(\gamma)(p_2)$ by~\eqref{G-formula}.

On the other hand, if $F$ is an embedding, then $G$ is surjective by~\eqref{G-formula}.

\section{Notation}\label{notation-section}

In the remainder of the paper, except where otherwise stated, $M$ and $N$
are smooth manifolds with $\dim(M)<\dim(N)$, $g$ is a smooth Riemannian metric on $N$,
and $F:M\to N$ is simple, smooth, $g$-minimal immersion.

We let $W$ be an open subset of $N$ such that $U:=F^{-1}(W)$ contains a point from each component
of $N$ and such that $F|_U$ is an embedding.  (Such a $W$ exists since $F$ is simple.)
If $M$ is connected, one can choose $W$ to be a small neighborhood of a point $p\in N$ such that
exactly one sheet of $F(M)$ passes through $p$.

As in \S\ref{mean-curvature-operator-section}, we let $\Vv=\Vv_F$ be the space
of smooth normal vectorfields on $F$.
We let $V_0$ be the set of $f\in \Vv$ such that $Jf$ is supported in $U$.
It may be helpful to think of $f\in V_0$ as ``almost" a Jacobi field: $Jf=0$ outside of the very small set $U$.

\section{families of immersions}\label{families-section}

\begin{proposition}\label{G-image}
If $f\in V_0$, then there is a $\gamma\in C^\infty(N)$ such that $G\gamma=f$.
\end{proposition}

\begin{proof}
This follows immediately from~\eqref{G-formula}.
\end{proof}

\begin{theorem}\label{families-theorem}
Suppose that $F:M\to N$ is a simple, smooth, $g$-minimal immersion with no nontrivial
Jacobi fields (i.e., no nonzero solutions $v\in \Vv_F$ of $Jv=0$.)
Let $U$, $W$ and $V_0$ be as in \S\ref{notation-section}.
Let $f_1,\dots,f_k$ be vectorfields in $V_0$.

There exist $\eps>0$ and smooth maps
\begin{align*}
   &\gamma: \BB^k(0,\eps)\to C^\infty(N), \\
   &\Ff: \BB^k(0,\eps)\times M\to N
\end{align*}
with the following properties:
\begin{enumerate}[\upshape (1)]
\item\label{initial-gamma-item} $\gamma(0,\cdot)=0$.
\item\label{initial-F-item} $\Ff(0,\cdot)=F(\cdot)$.
\item\label{minimal-item} For each $\tau\in \BB^k(0,\eps)$, the map $\Ff(\tau,\cdot):M \to N$ is an immersion
 that is minimal with respect to the metric $e^{\gamma(\tau)}g$.
\item\label{support-item} For each $\tau$, $\gamma(\tau)$ is supported in $W$.
\item\label{velocity-item} For each $i$,
\[
   (d/dt)_{t=0} \Ff(t\ee_i, \cdot) = f_i(\cdot),
\]
or, in other notation,
\[
   D_1\Ff(0,\cdot)\ee_i = f_i(\cdot).
\]
\end{enumerate}
\end{theorem}

\begin{remark}\label{families-remark}
Theorem~\ref{families-theorem} can be restated as follows.
Given a $k$-dimensional linear subspace $V$ of $V_0$, there exist smooth
maps $\gamma$ and $\Ff$ such that~\eqref{initial-gamma-item} -- \eqref{support-item} hold,
 and such that $v\mapsto D_1\Ff(0,\cdot)v$ is a surjective linear map from $\RR^d$
onto $V$.
\end{remark}

\begin{proof}
By Proposition~\ref{G-image}, for each $i$, we can find a $\gamma_i\in C^\infty(N)$ supported in $W$
such that $G\gamma_i=-Jv_i$.
Define $\gamma: \RR^k\times N\to \RR$ by
\[
   \gamma(\tau,\cdot) = \sum_{i=1}^k \tau_i \gamma_i(\cdot).
\]
By the implicit function theorem, there is an $\eps>0$ and a smooth map
\[
  \Ff:\BB^k(0,\eps)\times M \to N
\]
such that~\eqref{initial-gamma-item}--\eqref{minimal-item} hold. Note also that~\eqref{support-item}
holds by our choice of the $\gamma_i$.  Thus it remains only to show~\eqref{velocity-item}.

Since $\Ff(t\ee_i, \cdot)$ is $e^{t\gamma_i}g$-minimal,
\[
   0 = H(t\gamma_i, \Ff(t\ee_i, \cdot)).
\]
Taking the derivative at $t=0$ gives
\begin{align*}
0
&=  D_1H(0,0)\gamma_i + D_2H(0,0) (d/dt)_{t=0} \Ff(t\ee_i,\cdot) \\
&= G \gamma_i + J (d/dt)_{t=0} \Ff(t\ee_i,\cdot) \\
&= - Jv_i + J (d/dt)_{t=0}\Ff(t\ee_i, \cdot).
\end{align*}
Since $J$ has no nontrivial kernel, 
\[
  v_i = (d/dt)_{t=0}\Ff(t\ee_i,\cdot).
\]
\end{proof}

\section{First Transversality Theorem}

\begin{theorem}[Submersion Theorem]\label{submersion-theorem}
Suppose that $F:M\to N$ is a smooth, simple, $g$-minimal immersion with no nontrivial
Jacobi fields (i.e., no nonzero solutions $v\in \Vv_F$ of $Jv=0$.)

Then for some finite $k$ and some $\eps>0$, there is a smooth function $\gamma:\BB^k(0,\eps)\to C^\infty(N)$
and a smooth map $\Ff: \BB^k(0,\eps)\times M\to N$ such that
\begin{enumerate}
\item\label{initial-assertion} $\gamma(0)=0$ and $\Ff(0,\cdot)=F$.
\item\label{minimal-assertion} For each $\tau\in \BB^k(0,\eps)$, $\Ff(\tau,\cdot):M\to N$ is an smooth 
immersion that is minimal with respect to the metric $e^{\gamma(t)}g$.
\item\label{submersion-assertion} $\Ff$ is a submersion.
\end{enumerate}
\end{theorem}

\begin{proof}
Since $F:M\to N$ is simple, there is an open subset $W$ of $N$ such that 
 that $F$ is an embedding of $U:=F^{-1}(W)$ into $W$
 and such that $U$ has points in each connected component of $M$  

Let 
\[
    V_0 = \{ f\in \Vv_F: \text{$Jf$ is supported in $U$}\}.
\]

By a very general fact about solutions of linear PDEs on compact manifolds (see Theorem~\ref{PDE-theorem} below),
$V_0$ has a finite-dimensional subspace $V$ with the following property:
\begin{equation}\label{property}
\text{\parbox{.85\textwidth}{For every point $p\in M$ and normal vector $v\in \Tan(F,p)^\perp$,
there is an $f\in V$ such that $f(p)=v$.}}
\end{equation}
 By Theorem~\ref{families-theorem} and Remark~\ref{families-remark},
  there exist $\eps>0$ and smooth maps
\begin{align*}
  &\gamma:\BB^k(0,\eps) \to C^\infty(N),  \\
  &\Ff: \BB^k(0,\eps)\times M \to N
\end{align*}
such that Assertions~\eqref{initial-assertion} and~\eqref{minimal-assertion} hold and such that
\begin{equation}\label{onto}
  \text{$D\Ff_1(0,\cdot)$ is a surjection from $\RR^k$ onto $V$.}
\end{equation}

To complete the proof, we show that by replacing $\eps>0$ by a smaller
positive number $\eps'$ (and by replacing $\gamma$ and $\Ff$ by
their restrictions to $\BB^k(0,\eps')$ and $\BB^k(0,\eps')\times \Omega^kM$)
we can make $\Ff$ be a submersion.

Let $x\in M$, let $p=F(x)=\Ff(0,x)$, and let $v\in \Tan_pN$.
Write $v=v'+v''$ where $v'\in \Tan(F,x)^\perp$ and $v''\in \Tan(F,x)$.

By~\eqref{property}, there is an $f\in V$ such that $f(x)=v'$.

By~\eqref{onto}, there is a vector $\tau\in \RR^k$ such that 
\[
    D_1 \Ff(0,\cdot) \tau = f.
\]
Thus
\[
 D_1\Ff(0,x)\tau = f(x) = v'.
\]
Since $v''\in \Tan(F,x)$, there is a vector $\xi$ in $\Tan(M,x)$ such that
\[
   DF(x)\xi = v''.
\]
Since $\Ff(0,\cdot)=F(\cdot)$, we see that $D_2\Ff(0,x)=DF(x)$ and therefore
\[
   D_2\Ff(0,x)\xi = v''.
\]
Thus 
\begin{align*}
D\Ff(0,x)(\tau,\xi) 
&= D_1\Ff(0,x)\tau + D_2\Ff(0,x)\xi  \\
&= v' + v'' \\
&= v.
\end{align*}
We have shown that $D\Ff(0,x)$ is surjective for every $x\in M$.
Hence by replacing $\eps$ by a smaller $\eps>0$, we can guarantee that
$D\Ff$ is surjective at all points of $\BB^k(0,\eps)\times M$, i.e., that $\Ff$ is a submersion.
\end{proof}

\begin{theorem}\label{transversality-theorem}
Suppose that $N$ is a smooth manifold with a smooth Riemannian metric $g_0$,
and that $\Gamma$ is a smooth submanifold of $N$.
For a generic smooth metric $g$ conformal to $g_0$, the following holds:
if $F:M\to N$ is a simple, $g$-minimal immersion of a closed manifold $M$ into $N$,
then $F$ is transverse to $\Gamma$.
\end{theorem}

\begin{proof}
Let $F$ be a $g$-minimal immersion of a closed manifold $M$ into $N$ with no nontrivial Jacobi fields.
Let
\begin{align*}
  &\gamma:\BB^k(0,\eps)\to C^\infty(N), \, \text{and} \\
  &\Ff: \BB^k(0,\eps)\times M\to N
\end{align*}
be as in the Submersion Theorem (Theorem~\ref{submersion-theorem}).

Since $\Ff$ is a submersion, it is transverse to $\Gamma$.
Therefore (by the Parametric Transversality Theorem), $\Ff(\tau,\cdot):M\to N$
is transverse to $\Gamma$ for almost all $\tau$.

In particular, there is a sequence $\tau_i\to 0$ such that $\Ff(\tau_i,\cdot): M \to N$
is transverse to $\Gamma$.  
Theorem~\ref{transversality-theorem} now follows from the Bumpy Metrics Corollary~\ref{bumpy-corollary}.
(To see Corollary~\ref{bumpy-corollary} applies, let $Q$ be the set of points in $\MM$ corresponding to immersions
that are transverse to $\Gamma$.  Note that transversality is an open condition, so $Q$ is an open
subset of $\MM$.  Thus we can apply Corollary~\ref{bumpy-corollary} to the set $\Kk:=\MM\setminus Q$.)
\end{proof}

\begin{remark}\label{non-simple-remark}
Let $g$ be a metric satisfying the conclusion of Theorem~\ref{transversality-theorem}.
Then the conclusion also holds for non-simple immersions.
For let $F:M\to N$ be any $g$-minimal immersion of a closed manifold $M$ into $N$.
Using unique continuation, it is not hard to prove that $F$ factors through a simple immersion.
To be precise, there is a closed manifold $M'$, a simple immersion $F' : M' \to N$, and 
a covering map $\pi: M \to M'$ such that $F=F'\circ \pi$.
By hypothesis, $F'$ is transverse to $\Gamma$.  But then (trivially) $F$ is also transverse to $\Gamma$.
\end{remark}

\section{Strong Transversality}\label{strong-section}

\newcommand{\codimension}{\operatorname{codimension}}

If $S$ is a set, we let $\Delta^kS$ be the diagonal in $S^k$:
\[
    \Delta^kS = \{(x_1, x_2, \dots, x_k)\in S^k: x_1=x_2=\dots = x_k\},
\]
and we let
\[  
  \Omega^kS = \{ (x_1,\dots,x_k)\in S^k: \text{$x_i\ne x_j$ for every $i\ne j$} \}.
\]

\newcommand{\image}{\operatorname{image}}

\begin{definition}
Let $F:M\to N$ be a smooth map between smooth manifolds $M$ and $N$, let $\Gamma$ be a smooth submanifold 
of $N$, and let $k\ge 1$ be an integer.  We say that $F$ is {\bf $k$-transverse} to $\Gamma$
provided the map
\begin{align*}
&\widetilde F: \Omega^kM \to N^{k} \\
&\widetilde F(x_1,\dots,x_k) = (F(x_1), \dots, F(x_k))
\end{align*}
is transverse to $\Delta^k\Gamma$.
We say that $F$ is {\bf strongly transverse} to $\Gamma$ if it is $k$-transverse for every $k\ge 1$.
\end{definition}

Note that $1$-transversality is the same as transversality.
Theorem~\ref{more-intuitive-theorem} in Section~\ref{equivalence-section}
 gives a more geometrically intuitive description of strong transversality.

\begin{theorem}\label{multi-submersion-theorem}
Given $F$ as in \S\ref{notation-section} with no nontrivial Jacobi fields and a positive integer $k$,
 there exist $\eps>0$, $d<\infty$, and smooth maps
\begin{align*}
&\gamma: \BB^d(0,\eps) \to C^\infty(N), \\
&\Ff:\BB^d(0,\eps) \times M \to N
\end{align*}
with the following properties:
\begin{enumerate}[\upshape(1)]
\item\label{initial-again-item} $\gamma(0)=0$ and $\Ff(0,\cdot)=F(\cdot)$.
\item\label{minimal-again-item} For each $\tau$, the map $\Ff(\tau,\cdot):M\to N$ is a smooth immersion
that is minimal with respect to the metric $e^{\gamma(\tau)}g$.
\item\label{submersion-again-item} The map
\begin{align*}
&\widetilde\Ff: \BB^d(0,\eps)\times \Omega^kM  \to N^k, \\
&\widetilde \Ff(\tau,x_1,\dots, x_k) = (\Ff(\tau,x_1), \dots , \Ff(\tau,x_k))
\end{align*}
is a submersion at each point of $\Cc:=\widetilde \Ff^{-1}(\Delta^{k}N)$.
\end{enumerate}
\end{theorem}

\begin{proof}
Since $F$ is an immersion, the set
\[
  C: = \{(x_1,\dots,x_k)\in \Omega^kM: F(x_1)=F(x_2)=\dots = F(x_k) \}
\]
is a compact subset of $\Omega^kM$.

Hence, by a general PDE theorem (Theorem~\ref{PDE-theorem}) there is a finite-dimensional
subspace $V$ of $V_0$ with the following property:

\begin{equation}\label{multi-property}
\text{\parbox{.85\textwidth}{Given $(x_1,\dots,x_k)\in C$ and $v_i\in \Tan(F,x_i)^\perp$, there is 
a section $f\in V$ such that $f(x_i)=v_i$ for each $i\in \{1,\dots,k\}$.}}
\end{equation}
By Theorem~\ref{families-theorem}, 
there exist $\gamma$ and $\Ff$ such that~\eqref{initial-again-item} and~\eqref{minimal-again-item} hold
and such that
\begin{equation}\label{onto-again}
  \text{$D\Ff_1(0,\cdot)$ is a surjection from $\RR^d$ onto $V$.}
\end{equation}
It remains only to verify~\eqref{submersion-again-item}.

Define
\begin{align*}
&\rho: \RR^k\times \Omega^kM\to \RR, \\
&\rho(\tau,x_1,\dots, x_k) = |\tau|,
\end{align*}
and let
\begin{align*}
\Cc_0 
&= \rho^{-1}(0)\cap \Cc \\
&= \{(\tau,x_1,\dots,x_k): \text{$\tau=0$ and $(x_1,\dots,x_k)\in C$} \}.
\end{align*}

\begin{claim}\label{multi-submersion-claim}
$\widetilde\Ff$ is a submersion at each point $(0,x_1,\dots,x_k)$ of $\Cc_0$.
\end{claim}

\begin{proof}[Proof of claim]
Suppose $(0,x_1,\dots,x_k)\in \Cc_0$.
By definition of $\Cc_0$, 
\[
 F(x_1)=F(x_2)=\dots=F(x_k)=q
\]
 for some $q\in N$.
Let
\[
  (v_1,\dots,v_k)\in \Tan(N^k,(q,\dots,q)) = (\Tan(N,q))^k.
\]
Write $v_i=v_i'+v_i''$ where $v_i'\in \Tan(F,x_i)^\perp$
and $v_i''\in\Tan(F,x_i)$.

By~\eqref{multi-property}, there is an $f\in V$ such that $f(x_i)=v_i'$ for each $i$.
By~\eqref{onto-again}, there is a $\tau\in \RR^d$ such that
\[
   D_1\Ff(0,\cdot)\tau =  f.
\]
Thus
\[
   D_1\Ff(0,x_i)\tau = f(x_i) = v_i'.
\]

Since $v_i''\in \Tan(F,x_i)$, there is a vector $\xi_i\in \Tan(M,x_i)$ such that
\begin{equation*}
DF(x_i)\xi_i=v_i''.
\end{equation*}
Since $\Ff(0,\cdot)=F(\cdot)$,
\begin{equation}\label{xi_i}
D_2\Ff(0,x_i)\xi_i = DF(x_i)\xi_i = v_i''.
\end{equation}

Now
\begin{align*}
D\Ff(0,x_i)(\tau,\xi_i)
&=
D_1\Ff(0,x_i)\tau + D_2\Ff(0,x_i)\xi_i
\\
&=
v_i' + v_i''
\\
& = v_i.
\end{align*}
Consequently,
\[
D\widetilde\Ff(0,x_1,\dots,x_k)(\tau,\xi_1,\dots,\xi_k) = (v_1,\dots,v_k).
\]
This completes the proof of Claim~\ref{multi-submersion-claim}.
\end{proof}

Now let  $\Kk$ be the set of points in $\Cc$ where $\widetilde\Ff$ is not a submersion.
If $\Kk$ is empty, we are done.
Otherwise, note that $\Kk$ is compact (it is a closed subset of the compact set $\Cc$).  
By Claim~\ref{multi-submersion-claim}, $\rho(\cdot)>0$ at each point in $\Kk$.
Hence
\[
  \eps':= \min_{\Kk}\rho(\cdot)/2 > 0.
\]
Now replace $\eps$ by $\eps'$ (and therefore $\BB^d(0,\eps)$ by $\BB^d(0,\eps')$.)
\end{proof}

\begin{theorem}\label{strong-transversality-theorem}
Suppose that $N$ is a smooth manifold with a smooth Riemannian metric $g_0$,
and that $\Gamma$ is a smooth submanifold of $N$.
For a generic set of smooth metrics conformal to $g_0$, the following holds:
if $F:M\to N$ is a simple, $g$-minimal immersion of a closed manifold $M$ into $N$,
then $F$ is strongly transverse to $\Gamma$.
\end{theorem}

\begin{proof}
It suffices to show for each $k$ that the theorem holds with ``$k$-transversality" in place
of ``strong transversality".

Let $F$ be a $g$-minimal immersion of a closed manifold $M$ into $N$ with no nontrivial Jacobi fields.
Let
\begin{align*}
  &\gamma:\BB^d(0,\eps)\to C^\infty(N),  \\
  &\Ff: \BB^d(0,\eps)\times M\to N, \, \text{and} \\
  &\widetilde\Ff: \BB^d(0,\eps)\times \Omega^kM \to N^k
\end{align*}
be as in the Theorem~\ref{multi-submersion-theorem}.

Since $\widetilde{\Ff}$ is a submersion at all points of $\widetilde{\Ff}^{-1}(\Delta^kN)$,
it is  transverse to $\Delta^k\Gamma$.
By the Parametric Transversality Theorem, $\widetilde\Ff(\tau,\cdot):M\to N$
is transverse to $\Delta^k\Gamma$ for almost all $\tau$.
Therefore $\Ff(\tau,\cdot):M\to N$ is $k$-transverse to $\Gamma$ for almost all $\tau$.

In particular, there is a sequence $\tau_i\to 0$ such that $\Ff(\tau_i,\cdot): M \to N$
is $k$-transverse to $\Gamma$.  
Theorem~\ref{strong-transversality-theorem} (with ``$k$-transverse" in place of ``strongly transverse")
now follows from the Bumpy Metrics Corollary~\ref{bumpy-corollary}.
(Let $Q$ be the set of points in $\MM$ corresponding to immersions
that are $k$-transverse to $\Gamma$.  Since $k$-transversality is an open condition, $Q$ is an open
subset of $\MM$.  Thus we can apply Corollary~\ref{bumpy-corollary} to the set $K:=\MM\setminus Q$.)
\end{proof}

\begin{definition}
A smooth immersion $F:M\to N$ is called ``$k$-self-transverse" provided
$F$ is $k$-transverse to $N$.  We say that $F$ is {\bf strongly self-transverse}
if it is $k$-self-transverse for every $k\ge 1$.
\end{definition}

The immersion $F$ is $2$-self-transverse if and only if it is self-transverse as defined in the
introduction.

As the special case $\Gamma=N$ of Theorem~\ref{strong-transversality-theorem}, we have

\begin{theorem}\label{strong-selfie-theorem}
Suppose that $N$ is a smooth manifold with a smooth Riemannian metric $g_0$.
For a generic set  of smooth metrics $g$ conformal to $g_0$, the following holds:
if $F:M\to N$ is a simple, $g$-minimal immersion of a closed manifold $M$ into $N$,
then $F$ is strongly self-transverse.
\end{theorem}

{\color{blue}
\begin{corollary}\label{strong-selfie-corollary}
In Theorem~\ref{strong-selfie-theorem},
the following holds for a generic set  of smooth metrics $g$ conformal to $g_0$:
if $F:M\to N$ is a simple, $g$-minimal immersion of a closed manifold $M$ into $N$
and if $\dim(M)< \frac12\dim(N)$, 
then $F$ is an embedding.
\end{corollary}
}  

\section{The Geometry of Strong Transversality}\label{equivalence-section}


In this section, we give a more geometrically intuitive characterization of
   strong transversality of maps.
  Nothing in the section is required for
 the rest of the paper.
 
\begin{definition}\label{strongly-transverse-subspace-definition}
Let $V_1, \dots, V_k$ be linear subspaces of a finite-dimensional Euclidean space $V$,
We say that $V_1,\dots,V_k$ are {\bf  strongly transverse} provided the following holds:
if $v_i\in V_i^\perp$ and $\sum_iv_i=0$, then $v_1=\dots=v_k=0$.
\end{definition}

The following theorem gives equivalent characterizations of strong transversality of linear
subspaces.  In particular, it shows that whether $V_1,\dots,V_k$ are strongly transverse
does not depend on choice of the inner product on $V$.

\begin{theorem}\label{linear-equivalence-theorem}
Let $V_1, \dots, V_{k+1}$ be linear subspaces of a finite-dimensional Euclidean space $V$.
Let $\Pi_i:V\to V_i^\perp$ be the orthogonal projection, and let
\begin{align*}
&L: (\cap_iV_i)^\perp \to V_1^\perp \times \dots\times V_{k+1}^\perp, \\
&L(v) = (\Pi_1v,\dots, \Pi_{k+1} v).
\end{align*}
The following are equivalent:
\begin{enumerate}
\item\label{L-isomorphism}The map $L$ is an isomorphism.
\item\label{L-surjective} The map $L$ is surjective.
\item\label{dim-equality} $\dim((\cap_iV_i)^\perp)=\sum \dim(V_i^\perp)$.
\item\label{dim-inequality} $\dim((\cap_iV_i)^\perp) \ge \sum \dim(V_i^\perp)$.
\item\label{transverse-item} $V_1\times\dots\times V_k$ is transverse to $\Delta^kV_{k+1}$.  That is, 
\[ 
   V^k = (V_1\times V_2 \times\dots\times V_k) + \Delta^kV_{k+1}.
\]
\item\label{T-isomorphism} The map 
\begin{align*}
      &T: V_1^\perp\times\dots \times V_{k+1}^\perp \to (\cap_iV_i)^\perp, \\
      &T(v_1,\dots,v_{k+1})= v_1 + \dots + v_{k+1}
\end{align*}
is an isomorphism.
\item\label{T-injective} The map $T$ in~\eqref{T-isomorphism} is injective.
\item\label{rephrase} The spaces $V_1,\dots,V_{k+1}$ are strongly transverse in $V$.
\end{enumerate}
\end{theorem}

From Condition~\eqref{dim-equality}, it is easy to check that $V_1, V_2$ are strongly transverse
 if and only if $V_1+V_2=V$.  Thus in the case of a pair
of subspaces, strong transversality and transversality are the same.

\begin{proof}[Proof of Theorem~\ref{linear-equivalence-theorem}]
That $L$ is injective follows immediately from its definition.
Thus the following are equivalent: 
(i) $L$ is an isomorphism, (ii) $L$ is surjective, (iii) the dimensions of the domain and range are equal,
(iv) the dimension of the domain is $\ge$ the dimension of the target.
Now the dimensions of the domain and target are the right and left sides of the equation in~\eqref{dim-equality}.
This proves the equivalence of~\eqref{L-isomorphism}--\eqref{dim-inequality}.

To show that~\eqref{L-surjective} implies~\eqref{transverse-item},
 assume that~\eqref{L-surjective} holds and let $(v_1,\dots,v_k)\in V^k$.
Then $(v_1,\dots,v_k,0)\in V_1\times\dots\times V_{k+1}$, 
so by~\eqref{L-surjective}, there is a vector $v\in (\cap_iV_i)^\perp$ such that
\begin{align*}
    \Pi_iv &= \Pi_iv_i \quad\text{for $i\le k$, and} \\
    \Pi_{k+1}v &= 0.
\end{align*}
Thus $v\in V_{k+1}$, and, for $i\le k$, $\Pi_i(v_i-v)=0$, so $v_i-v\in V_i$.
Consequently,
\begin{align*}
(v_1,\dots,v_k)
&=
(v_1-v, v_2-v,\dots,v_k-v) + (v,v,\dots,v) \\
&\in
(V_1\times V_2 \times\dots\times V_k) + \Delta^kV_{k+1}.
\end{align*}
This completes the proof that~\eqref{L-surjective} implies~\eqref{transverse-item}.

Now suppose that~\eqref{transverse-item} holds 
and let $(v_1,\dots,v_{k+1})\in V_1^\perp\times\dots\times V_{k+1}^\perp$.
Then $(v_1,\dots,v_k)\in V^k$, so 
by~\eqref{transverse-item}, there exists $u_i$ in $V_i$ (for $i\le k$) and $u\in V_{k+1}$ such that
\[
 (v_1,\dots, v_k) = (u_1,\dots,u_k) + (u ,\dots,u).
\]
Thus $v_i=u_i+u$ for $i\le k$, so, applying $\Pi_i$ gives
\begin{equation}\label{sagaroon}
  v_i = \Pi_i u \quad (i\le k)
\end{equation}
since $v_i\in V_i^\perp$ and $u_i\in V_i$.  Also by (5), there exist $w_i\in V_i$ ($i\le k$) and $w \in V_{k+1}$
such that
\[
  (v_{k+1}, \dots,v_{k+1}) = (w_1,\dots,w_k) + (w,\dots,w).
\]
Thus $v_{k+1}=w_i + w$ for $i\le k$, so applying $\Pi_i$ gives
\begin{equation}\label{lionel}
    \Pi_i v_{k+1} = \Pi_i w \quad (\text{for $i\le k$}).
\end{equation}
Now let 
\begin{equation}\label{v-formula}
\sigma =u + v_{k+1} - w.
\end{equation}
Applying $\Pi_i$ gives
\begin{equation}\label{OK-to-k}
   \Pi_i \sigma = v_i \quad\text{for $i\le k$},
\end{equation}
by~\eqref{sagaroon} and~\eqref{lionel}.
Applying $\Pi_{k+1}$ to~\eqref{v-formula} gives
\begin{equation}\label{OK-beyond}
  \Pi_{k+1}\sigma = v_{k+1}
\end{equation}
since $u$ and $w$ are in $V_{k+1}$.

  Now let $v = \sigma - \Pi \sigma$, where $\Pi: V\to \cap_iV_i$ is the orthogonal projection.
Then $v \in (\cap_i V_i)^\perp$ and $\Pi_iv=v_i$ for all $i\le k+1$ by~\eqref{OK-to-k} and~\eqref{OK-beyond}.
Thus~\eqref{L-surjective} holds.  This completes the proof that~\eqref{transverse-item} implies~\eqref{L-surjective}.

Note that~\eqref{T-isomorphism} and~\eqref{T-injective} 
are equivalent to~\eqref{L-isomorphism} and~\eqref{L-surjective}, respectively, because $T$ is the adjoint of the map $L$.
Finally, \eqref{T-injective} and~\eqref{rephrase} are trivially equivalent.
\end{proof}

\begin{theorem}\label{more-intuitive-theorem}
Suppose that $F:M\to N$ is a smooth immersion, that $\Gamma$ is a smooth submanifold of $N$,
and that $p$ is a point in $\Gamma$.
Let $V_1,V_2,\dots, V_k$ be the tangent planes to the sheets of $F:M\to N$ that
pass through $p$.   Then 
\begin{enumerate}  
\item $F$ is strongly transverse to $\Gamma$ at $p$ if and only $V_1,V_2,\dots, V_k ,\Tan(\Gamma,q)$
   are strongly transverse subspaces of $\Tan(N,q)$.
\item $F$ is strongly self-transverse at $p$ if and only if $V_1,V_2,\dots,V_k$ are strongly transverse subspaces of
   $\Tan(N,p)$.
\end{enumerate}
\end{theorem}

\begin{proof}
This follows immediately from Condition~\eqref{transverse-item} in Theorem~\ref{linear-equivalence-theorem}.
\end{proof}

In Theorem~\ref{more-intuitive-theorem}, $F$ can actually be any smooth map (not necessarily an immersion).
In the general case, we let $x_1,\dots,x_k$ be the points of $F^{-1}(p)$ and (for each $i$) 
we let  $V_i$ be the image of $\Tan(M,x_i)$ under $DF(x_i)$.

\section{Ample Spaces}\label{ample-section}

In the next section, we prove the general theorems about linear PDE that were
needed in the proof of the main theorem (Theorem~\ref{strong-transversality-theorem}).
  In this section, we present a few preliminaries
results about spaces of sections of vector bundles.

Throughout this section, we fix a $d$-dimensional vector bundle over a Hausdorff space $M$.
If $S\subset M$, we let $K(S)$ denote the space of continuous sections over $S$,
i.e., the set of continuous maps that assign to each $x\in S$ a vector in the fiber at $x$.

Of course if the bundle is trivial over $S$, then $K(S)$ may be identified with $C^0(S,\RR^d)$.
Though we need the results of this section for vector bundles that may not be trivial,
the proofs are the same whether or not the bundle is trivial.  
Thus there is no real loss of generality if this section is read with a trivial bundle in mind.
That is, wherever we write $K(S)$, the reader may think $C^0(S,\RR^d)$.

Note that if $S\subset M$ is finite set with $k$ elements, then $K(S)$ is a finite-dimensional vector space,
since $K(S)\cong C^0(S,\RR^d) \cong \RR^{kd}$.

If $X\subset Y\subset M$ and if $V$ is a linear subspace of $K(Y)$, we
let 
\[
   V|_X = \{ f|_X: f\in V\},
\]
where $f|_X$ denotes the restriction of $f$ to $X$.

\begin{definition}
Suppose that $X\subset M$ and that $V$ is a linear subspace of $K(X)$.
Let $S$ be a finite subset of $X$.   We say that $V$ is {\bf ample} for $S$ if
\[
   V|_S = K(S).
\]
\end{definition}

In other words, $V$ is ample for $\{p_1, \dots, p_k\}$ provided the following
holds: given vectors $v_i$ at $p_i$, there is an $f\in V$ such that $f(p_i)=v_i$ for each $i\in \{1,2,\dots,k\}$.

We will also write  ``$V$ is ample for the point $p$"
to mean ``$V$ is ample for $\{p\}$".

Note that if $S$ is a finite set with more than one point, 
then being on ample for $S$ is much stronger than being ample for each point of $S$.
For example, suppose that the bundle is trivial, and 
let $V$ be the space of constant functions in $K(M)$.  Then $V$ is ample for each point of $M$,
but $V$ is not ample for any set with more than one point.

\begin{theorem}
Suppose that $X\subset M$, that $V$ and $V'$ are linear subspaces of $K(X)$, 
 and that $S$ is a finite subset of $X$.
 If $V$ is ample for $S$ and if $V'$ is dense in $V$, then $V'$ is ample for $S$.
 \end{theorem}

\begin{proof}
Since $V'$ is dense in $V$ and since $V|_S = K(S)$, we see that $V'|_S$ is dense
in $K(S)$.  But since $K(S)$ is finite-dimensional, a subspace of $K(S)$ that is dense
in $K(S)$ must be all of $K(S)$.
\end{proof}

\begin{theorem}\label{ample-compact-point-theorem}
Suppose that $C\subset M$ is compact, that $V_0$ is a linear subspace of $K(M)$,
and that $V_0$ is ample for every $p\in C$.  Then $V_0$ has a finite-dimensional
subspace $V$ such that $V$ is ample for every $p\in C$.
\end{theorem}

\begin{proof}
Let $\Vv$ be the collection of all finite-dimensional subspaces of $V_0$.
For each $V\in \Vv$, let $U(V)$ be the set of points $x$ such that $V$ is ample at $x$.
Note that $U(V)$ is an open set.

If $p\in C$, then $V_0$ is ample for $p$, from which it trivially follows that $V_0$
has a finite-dimensional subspace that is ample at $p$.  Consequently, 
\[
   C \subset \cup_{V\in \Vv} U(V).
\]
Since $C$ is compact, this open cover has a finite subcover:
\[
  C \subset U(V_1)\cup \dots \cup U(V_k).
\]
Now let $V$ be $V_1+\dots+V_k$.
\end{proof}

\begin{theorem}\label{ample-compact-k-points-theorem}
Let $k$ be a positive integer and $C$ be a compact collection of $k$-element subsets of $M$.
Suppose $V_0$ is a linear subspace of $K(M)$ such that $V_0$ is ample 
for every $S\in C$.  Then $V_0$ contains a finite-dimensional subspace $V$
such that $V$ is ample for every $S\in C$.
\end{theorem}

Here the topology on $C$ is the obvious one. (It is the topology that comes from
identifying the space of $k$-element subsets of $M$ with a subset of  the quotient of $M^k$ by
the action of the permutation group $S_k$.)

Note that Theorem~\ref{ample-compact-point-theorem}
is the special case $k=1$ of Theorem~\ref{ample-compact-k-points-theorem}.

\begin{proof}
Theorem~\ref{ample-compact-k-points-theorem}
can be proved almost exactly as Theorem~\ref{ample-compact-point-theorem} was proved.
Alternatively, we can deduce Theorem~\ref{ample-compact-k-points-theorem} 
from Theorem~\ref{ample-compact-point-theorem} as follows.
For notational simplicity, assume that the vector bundle over $M$ is trivial,
so $K(M)=C^0(M,\RR^d)$.

Let $\widehat C$ be the set of $(p_1,\dots,p_k)\in M^k$ such that $\{p_1,\dots, p_k\}\in C$.

If $f\in C^0(M, \RR^d)$, let $\widehat f\in C^0(M^k, (\RR^d)^k)$ be given by
\[
    \{ \widehat f(p_1,\dots,p_k) = (f(p_1), f(p_2), \dots, f(p_k)) \}.
\]
Let 
\[
  \widehat V = \{ \widehat f: f\in V\}.
\]

Note that if $S=\{p_1, \dots, p_k\}$ is a $k$-element subset of $M$, then $f$ is ample for $S$
if and only if $\widehat f$ is ample at the point $(p_1,\dots,p_k)$.

Now apply Theorem~\ref{ample-compact-point-theorem} to the linear space $\widehat V$
 and the set $\widehat C$.
\end{proof}

\section{PDE}

In this section, we assume that $M$ is a smooth, closed, connected Riemannian manifold.
We consider some fixed smooth vector bundle over $M$ endowed with a smooth inner product on the
fibers.
Let $\Vv$ be the space of all smooth sections of the vector bundle.
If $D$ is a subdomain of $M$, we let $\Vv(D)$ be the space of smooth sections whose domain is $D$.
In the notation of Section~\ref{ample-section},
\[
   \Vv(D) = \{f\in K(D): \text{$f$ is smooth} \}.
\]

Let $J:\Vv\to\Vv$ be a second order, linear elliptic differential operator.
We assume that $J$ has the unique continuation property:

\begin{definition}
We say that $J$ has the {\bf unique continuation property} provided the following holds.
If $f\in \Vv$, if $Jf=0$ on a connected open set $U\subset M$, and if $Jf$ vanishes to infinite order at a point $p\in U$,
then $f$ vanishes everywhere in $U$.
\end{definition}

For the $J$ that arise in this paper, $J$ is the Laplacian plus lower order terms.  
In that case, it is well-known that $J$ has the unique continuation property. 
For example, it may be proved using Almgren's frequency function (as in~\cite{garofalo-lin}),
or by a theorem of Calderon~\cite{calderon}*{Theorem~11}.

In~\cite{lax}*{p.~760}, Peter Lax proved that the unique continuation property is equivalent to what he called the Runge Property:

\begin{definition}
We say that $J$ has the {\bf Runge Property} provided the following holds.
If $D_1\subset D_2$ are smooth closed domains in $M$ with nonempty boundaries such that $D_1$ is in the interior
of $D_2$ and such that $D_2 \setminus D_1$ is connected, then 
\[
    \{f|_{D_1}: \text{$f\in \Vv(D_2)$ and $Jf=0$} \}
\]
is dense (in the $L^2$ norm) in 
\[
  \{f \in \Vv(D_1): Jf=0\}.
\]
\end{definition}

Of course this $L^2$ denseness implies by elliptic regularity that if $C$ is a compact subset of the interior of $D_1$,
then the space
\begin{equation}\label{finite-dim-space-2}
   \{f|_{C}: \text{$f\in \Vv(D_2)$ and $Jf=0$} \}
\end{equation}
is dense with respect to smooth convergence in
\begin{equation}\label{finite-dim-space-1}
  \{f|_{C}: \text{$f\in \Vv(D_1)$ and $Jf=0$} \}.
\end{equation}
In particular, 
\begin{equation}\label{sameness}
  \text{If $C\subset D_1\setminus \partial D_1$ is finite, then the spaces~\eqref{finite-dim-space-2}
  and~\eqref{finite-dim-space-1} are the same}.
\end{equation}

\begin{theorem}[Runge-Type Theorem]\label{runge-type-theorem}
Let $U$ be a nonempty open subset of $M$ (e.g., a small open ball around a point)
and let 
\[
   V = \{ f\in \Vv: \text{$Jf$ is supported in $U$} \}.
\]
Suppose that $S$ is a finite subset of $M$ and that $\phi\in \Vv$.
Then there is an $f\in V$ such that $f(p)=\phi(p)$ and $Df(p)=D\phi(p)$ for each $p\in S$.
\end{theorem}

In other words, there exists an $f\in V$ with any prescribed values and prescribed first derivatives
at the points of $S$.

\begin{proof}
Let us assume for notational convenience 
that the bundle is trivial, and thus that $\Vv=C^\infty(M,\RR^d)$ for some $d$.
Let $A$ be the finite dimensional vector space
\[
  A = \oplus_{p\in S} ( \Tan(M,p) \oplus L(\Tan(M,p),\RR^d)).
\]
If $W$ is an open subset of $M$ containing $S$ and if $f\in C^\infty(W,\RR^d)$, we
define $\lambda(f)\in A$ by
\[
   \lambda(f) = (f(p), Df(p))_{p\in S}.
\] 
Thus the assertion of the Lemma is that $\lambda$ maps $V$ surjectively to $A$.

Let $q$ be a point  in $U\setminus S$.
Choose $r>0$ small enough that the closed geodesic balls of radius $r$ around the points in $S\cup\{q\}$ are
disjoint and diffeomorphic to closed balls in $\RR^m$,
  and so that $\overline{\BB(q,r)}$ is contained in $U$.

By standard PDE (see Lemma~\ref{local-lemma} below), we can choose $r>0$ sufficiently small that
for each $p\in S$, there exists an $f\in C^\infty(\BB(p,r), \RR^d)$ with $Jf=0$
having any prescribed values of $f(p)$ and $Df(p)$.   Thus if we let $D_1=\cup_{p\in S}\BB(p,r)$,
then 
\begin{equation}\label{D1-surjection}
 \text{$\lambda: \{f\in C^\infty(D_1,\RR^d): Jf = 0\} \to A$ is a surjective linear map}.
\end{equation}
Let $D_2 = M\setminus \BB(q,r/4)$.   By the Runge Property, 
\[ 
   \{ f|D_1: f\in C^\infty(D_2,\RR^d), \, Jf=0\}
\]
is dense in 
\[
  \{ f \in C^\infty(D_1,\RR^d): Jf = 0\}.
\]
Thus since $A$ is finite dimensional, it follows from~\eqref{D1-surjection} that
\begin{equation}\label{D2-surjection}
 \text{$\lambda: \{f\in C^\infty(D_2,\RR^d): Jf = 0\} \to A$ is a surjective linear map}.
\end{equation}
Let $\alpha \in A$.  Then there exists an $f\in C^\infty(D_2,\RR^d)$ such that $Jf=0$
and such that $\lambda(f)=\alpha$.
Let $\tilde f\in C^\infty(M,\RR^d)$ be a smooth function such that $\tilde f - f$ vanishes
outside of $\BB(q,r/2)$.  Then $\tilde f\in V$ and $\lambda(\tilde f) = \lambda(f)=\alpha$.
\end{proof}

\begin{theorem}\label{PDE-theorem}
Let $U$ be a nonempty open subset of $M$ (e.g., a small open ball around a point)
and let 
\[
   V = \{ f\in \Vv: \text{$Jf$ is supported in $U$} \}.
\]
Then
\begin{enumerate}[\upshape(1)]
\item\label{ample-assertion-1} For every finite subset $S$ of $M$, $V$ is ample for $S$.
\item\label{ample-assertion-2} If $k$ is positive integer and $C$ is a compact collection of $k$-element subsets of $M$,
  then there is a finite-dimensional subspace $\widehat{V}$ of $V$ such that $\widehat{V}$ is
  ample for every $S\in C$.
\item\label{ample-assertion-3}There is a finite-dimensional subspace $\widetilde V$ of $V$ 
such that $\widetilde{V}$ is 
   ample for every point in $M$.
\end{enumerate}
\end{theorem}

\begin{proof}
Assertion~\eqref{ample-assertion-1} follows from the Runge-Type Theorem~\ref{runge-type-theorem}.
(In fact, the Runge-Type Theorem is much stronger: it asserts that we can prescribe the
values and the first derivatives of $f\in V$ at the points in $S$, whereas
 Assertion~\eqref{ample-assertion-1} only asserts
that we can prescribe the values.)
Assertion~\eqref{ample-assertion-2} follows 
from Assertion~\eqref{ample-assertion-1} by Theorem~\ref{ample-compact-k-points-theorem}.
Assertion~\eqref{ample-assertion-3} is the special case of Assertion~\eqref{ample-assertion-2}
 when $k=1$ and $C=M$.
\end{proof}

\begin{lemma}\label{local-lemma}
Suppose that 
$J:C^\infty(\Omega,\RR^d)\to C^\infty(\Omega,\RR^d)$ is a 2nd-order linear elliptic operator,
where $\Omega\subset \RR^m$ is an open set containing the origin.  For all sufficiently small $r>0$,
the following holds.
For every vector $v\in \RR^d$ and for every linear map $L:\RR^m\to\RR^d$, there is a solution $f$ of $Jf=0$ on
 $\BB(0,r)$ such that $f(0)=v$ and such that $Df(0)=L$.
\end{lemma}

\begin{proof}
Choose $R>0$ so that for every $r\le R$  and for every affine function $\phi: \RR^m\to \RR^d$, 
there is a unique solution $f_{r,\phi}:\BB(0,r)\to \RR^d$
to the boundary value problem
\begin{align*}
  Jf_{r,\phi} &= 0,  \\
  f_{r,\phi} |\partial \BB(0,r)  &= \phi | \partial \BB(0,r).
\end{align*}
Let $\tilde{f}_{r,\phi}$ be the rescaled function
\[
\tilde{f}_{r,\phi}:   x\in \BB(0,1)\mapsto f_{r,\phi}(0) + r^{-1} (f_{r,\phi}(rx) - f_{r,\phi}(0)).
\]
Note that 
\begin{equation}\label{taylor}
 \text{$\tilde{f}_{r,\phi}(0)=f_{r,\phi}(0)$ and $D\tilde{f}_{r,\phi}(0)=Df_{r,\phi}(0)$},
\end{equation}
and that $\tilde{f}_{r,\phi}=\phi$ on $\partial \BB(0,1)$.

As $r\to 0$, $\tilde{f}_{r,\phi}$ converges
converge smoothly to the solution $\tilde{f}_\phi:\BB(0,1)\to\RR^m$ of 
\begin{align*}
&J_0\tilde{f}_\phi= 0, \\
&\tilde{f}_\phi|\partial \BB(0,1) = \phi|\partial\BB(0,1).
\end{align*}
where $J_0$ is a constant-coefficient, homogeneous, 2nd order linear elliptic operator.
Thus $\tilde{f}_\phi = \phi$.

Let $A$ be the space of affine maps from $\RR^m$ to $\RR^d$.
Since the linear map 
\begin{equation}\label{linear-maps}
   \phi \in A  \mapsto (\tilde{f}_{r,\phi}(0), D\tilde{f}_{r,\phi}(0)) \in \RR^d \times L(\RR^m,\RR^d)
\end{equation}
converges as $r\to 0$ to the linear bijection
\[
   \phi \in A \mapsto (\phi(0), D\phi(0)) \in \RR^d\times L(\RR^m,\RR^d),
\]
it follows that the map~\eqref{linear-maps} must be a bijection
for all sufficiently small $r$.  Thus we are done by~\eqref{taylor}.
\end{proof}

\begin{remark} (Not needed in this paper.)
A slight modification of the proof of Lemma~\ref{local-lemma} shows
that for each positive integer $k$, the following holds for all sufficiently small $r>0$.
If $\phi:\RR^m\to \RR^d$ is a polynomial of degree $k$ such that $J_0\phi=0$, then
there is an $f:\BB^m(0,r)\to \RR^d$ such that $Jf=0$ and such that $\phi$ is the degree~$k$
Taylor polynomial for $f$ at $0$.
\end{remark}

\section{Prescribed Mean Curvature Hypersurfaces}\label{prescribed-section}

The theorems in this paper easily extend to hypersurfaces with constant mean curvature or,
more generally, with prescribed mean curvature.
In those settings, one works with oriented surfaces.
 We say that two immersions $F_i:M_i\to N$
of smooth, oriented manifolds $M_1$ and $M_2$ are {\bf equivalent} if there
is an orientation-preserving diffeomorphism $u:M_1\to M_2$ such that $F_1=F_2\circ u$.
We let $[F]$ denote the equivalence class of $F$.

Now suppose that $N$ is an oriented smooth $(m+1)$-dimensional manifold with Riemmanian
metric $g$, and that $h$ is a smooth function on $N$.
Suppose that $F:M\to N$ is an immersion of an oriented $m$-manifold into $N$.
We say that $F$ has {\bf prescribed mean curvature $h$ with respect to the metric $g$} 
provided the mean curvature vector
at $x\in M$ is given by $-h(F(x)) \nu_F(x)$ where $\nu_F(x)$ is the unit normal to $\Tan(F,x)$
corresponding to the orientations of $N$ and of $M$.

 If we linearize the prescribed mean curvature equation about a critical point, we get the
$(g,h)$-Jacobi operator $J$.   As in the minimal case, if we restrict $J$ to the normal bundle,
 it is a self-adjoint, second-order,
linear elliptic operator whose leading term is the Laplacian.
As in the minimal case, solutions of $Ju=0$ are called $(g,h)$-Jacobi fields, or just Jacobi fields
if the $g$ and $h$ are understood.

\begin{theorem}\label{prescribed-banach-theorem}
Let $M$ be a smooth, closed, oriented $m$-dimensional manifold.
Let $N$ be a smooth, oriented, $(m+1)$-dimensional manifold with a smooth Riemannian metric $g$.
Let $h$ be a smooth function on $N$.
Let $q$ and $j$ be positive integers with $q>j\ge 2$ and let $\alpha\in (0,1)$.
Let $\MM^h(q,j,\alpha)$ be the set of pairs $(\gamma,[F])$ where $\gamma$ is a $C^q$ function
on $N$ and $F:M \to N$ is a simple, $C^{j,\alpha}$ immersion that has prescribed mean curvature $h$
with respect to the metric $e^{\gamma}g$.
Then $\MM^h(q,j,\alpha)$ is a separable, $C^{q-j}$ Banach manifold and the map
\begin{align*}
&\Pi: \MM^h(q,j,\alpha)\to C^q(N), \\
&\Pi(\gamma,[F]) = \gamma
\end{align*}
is a $C^{q-j}$ Fredholm map of Fredholm index $0$.
The kernel of $D\Pi(\gamma,[F])$ is naturally isomorphic to the
space of normal $(e^{\gamma}g,h)$-Jacobi fields to $[F]$.
\end{theorem}

In case $h$ is constant, this is proved in \cite{white-space}*{\S7}.
For functions $h$, the same proof works, except that in the first equation on
page~198, one replaces 
\[
     h \int u^{\#}\omega_{\gamma}
\]
by
\[
    \int u^{\#}(h\omega_{\gamma}),
\]
and similarly for the other formulas on that page.

\begin{corollary}\label{prescribed-bumpy-corollary}
In Theorem~\ref{prescribed-banach-theorem}, the set of critical values of $\Pi$ is meager in $C^q(N)$.
\end{corollary}

The Corollary follows from Theorem~\ref{prescribed-banach-theorem}  and 
the Sard-Smale Theorem~\cite{sard-smale}*{1.3}.

Let $N$, $g$, and $h$ be as in Theorem~\ref{prescribed-banach-theorem}.
Let $\MM^h$ be the space of all pairs $(\gamma, [F])$ such that $\gamma\in C^\infty(N)$
and $F$ is a smooth, simple, immersion with prescribed mean curvature $h$ with respect
to the metric $e^{\gamma}g$ from some closed, oriented $m$-manifold into $N$. 
Let $\Pi:\MM^h \to C^\infty(N)$ be the projection onto the first factor:  
\[
\Pi(\gamma,[F])=\gamma.
\]
Let $\MMreg^h$ be the union of open sets $U\in \MM^h$ such that $\Pi$ maps
$U$ homeomorphically onto an open subset of $C^\infty(N)$.
It follows from the implicit function theorem that $(\gamma,[F])\in \MMreg^h$
if and only if $[F]$ has no nonzero normal $(e^{\gamma},h)$-Jacobi fields.

Let $\MMsing^h= \MM^h\setminus \MMreg^h$.

\begin{theorem}\label{prescribed-bumpy-theorem}
The set $\Pi(\MMsing^h)$ is a meager subset of $C^\infty(N)$.
\end{theorem}

The paper~\cite{white-bumpy} proves that Theorem~\ref{prescribed-bumpy-theorem}
follows from Corollary~\ref{prescribed-bumpy-corollary} in the case $h=0$,
 but the proof given there works equal well for arbitrary $h$.

\begin{theorem}\label{main-theorem-prescribed}
Suppose that $N$ is a smooth, oriented, $(m+1)$-dimensional manifold with a smooth Riemannian metric $g_0$, that
$h$ is a smooth function on $N$, and
that $\Gamma$ is a smooth submanifold of $N$.
For a generic (in the sense of Baire category) smooth metric $g$ conformal to $g_0$,
if $F$ is any simple immersion of a closed, oriented $m$-manifold into $N$ that has prescribed mean curvature $h$
with respect to $g$, then
\begin{enumerate}[\upshape (1)]
\item   $F$ is strongly transverse to $\Gamma$, and
\item  $F$ is strongly self-transverse.
\end{enumerate}
\end{theorem}

Given Theorem~\ref{prescribed-bumpy-theorem}, the proof of Theorem~\ref{main-theorem-prescribed}
is exactly as in the minimal case.

(Theorem~\ref{prescribed-bumpy-theorem} is about a given closed $m$-manifold, whereas
Theorem~\ref{main-theorem-prescribed} is an assertion about all closed $m$-manifolds.
Note that for assertions about Baire Category, it does not matter whether or not one fixes the domain
manifold, since there are only countably many diffeomorphism types of smooth, closed $m$-manifolds.)

The following special case of Theorem~\ref{main-theorem-prescribed}
 is important in Xin Zhou's proof~\cite{zhou-multiplicity} of the multiplicity-one conjecture:

\begin{corollary}
Suppose in Theorem~\ref{main-theorem-prescribed} that $h^{-1}(0)$ is smoothly embedded
$m$-manifold in $N$.
For a generic (in the sense of Baire category) smooth metric $g$ conformal to $g_0$,
if $F$ is any simple immersion of a closed, oriented $m$-manifold into $N$ that has prescribed mean curvature $h$
with respect to $g$, then $F$ is transverse to $h^{-1}(0)$.
\end{corollary}

\begin{remark}
In this section, we have been assuming that $N$ and $M$ are orientable.
Actually, such orientations are not necessary: it suffices for the immersions
one works with to have oriented normal bundles.  That is, we work with
immersions $F:M\to N$ that are equipped with nowhere vanishing sections of the normal bundle.
With minor changes to the definitions, all the results in this section remain true (with the 
same proofs) in that 
slightly more general setting.
\end{remark}

\section{Generic Regularity of 2-Dimensional Locally Area-Minimizing Cycles}
\label{regularity-section}

In this section, we prove

\begin{theorem}\label{first-regularity-theorem}
For a generic smooth Riemannian metric $g$ on a manifold $N$, if $T$
  is a $2$-dimensional locally $g$-area-minimizing
 integral
cycle in $N$, then the support of $T$ is a smoothly embedded submanifold.
\end{theorem}

At the end of this section, we prove the analogous result for flat chains mod $2$.

Here, ``integral cycle" means ``integral current with boundary $0$", and 
 ``$T$ is locally area-minimizing" means that each point of $N$ has a neighborhood $\BB$ such
that the area (i.e., mass) of $T\llcorner \BB$ is less than or equal to the area of $T'$ for any integral current $T'$
in $\BB$ such that $\partial T'=\partial (T\llcorner\BB)$.  In particular, if $T$ minimizes area in its homology class, then it is locally
area-minimizing.  Thus we have

\begin{corollary}\label{first-regularity-corollary}
For a generic smooth Riemannian metric $g$ on a manifold $N$, if $T$ is a $2$-dimensional
 integral cycle in $N$ that minimizes $g$-area in its integral homology class, then the support of $T$ is a smoothly embedded submanifold.
\end{corollary}

If the dimension of $N$ is less than $4$, then the support of any locally area minimizing
integral cycle (for an arbitrary smooth ambient metric) is a smoothly embedded submanifold.
  Thus we will assume throughout this section that $\dim(N)\ge 4$.

De Lellis, Spadaro, and Spolaor~\cite{DeLellis-1, DeLellis-2, DeLellis-3},
 building on earlier work of Sheldon Chang~\cite{Chang}, proved that if $T$ is a $2$-dimensional
locally area-minimizing integral cycle, then the support of $T$ is a branched minimal surface.
Thus there is a closed (not necessarily connected) $2$-manifold $M$ and 
a branched minimal immersion
\[
    F: M \to N
\]
such that $F$ is simple and such that $F(M)$ is the support of $T$.

(If $M$ has connected components $M_1, M_2,\dots, M_k$, then $T$
will be the current 
\[
   \sum_i n_i F_\#[M_i]
\]
for some positive integers $n_1, \dots, n_k$.)

According to a theorem of J. D. Moore~\cite{moore-bumpy, moore-bumpy-correction},
 for a generic smooth metric on $N$,
every simple branched minimal immersion into $N$ is in fact an immersion (that is, free of branch points).
Consequently, for such a metric, every area-minimizing integral cycle $T$
has support equal to $F(M)$ for a simple minimal immersion $F:M\to N$
of a closed $2$-manifold $M$ into $N$.

Thus Theorem~\ref{first-regularity-theorem} follows from

\begin{theorem}
A generic smooth Riemannian metric $g$ on $N$ has the following property.
If $F$ is a simple $g$-minimal immersion of a closed $2$-manifold $M$
 into $N$ and if $F(M)$ is the support of a locally
area-minimizing integral cycle, then $F$ is an embedding.
\end{theorem}

In fact, we will prove a somewhat stronger result:

\begin{theorem}
Let $g_0$ be a smooth Riemannian metric on $N$.
A generic smooth metric $g$ conformal to $g_0$ has the following property.
If $F:M\to N$ is a simple $g$-minimal immersion of a closed $2$-manifold
 $M$ into $N$ and if $F(M)$ is the support of a locally
$g$-area-minimizing integral cycle, then $F$ is an embedding.
\end{theorem} 

\begin{proof}
If $\dim(N)>4$, the theorem is true
even without the condition ``and if $F(\Sigma)$ is the support of an area-minimizing
integral cycle"; see Corollary~\ref{strong-selfie-corollary}.
Thus we may assume that $N$ is a $4$-manifold.

If $\Sigma$ is the support of a $2$-dimensional locally area-minimizing
 integral cycle
in the $4$-manifold $N$ and if $P$ and $P'$ are planes in $\Tan(N,p)$ that are tangent to $\Sigma$ at a point $p\in N$,
then the pair $(P,P')$ has the following property~\cite{morgan}:
\begin{equation}\label{planes-property}\tag{*}
\parbox{11cm}{There is a complex structure $J$ on $\Tan(N,p)$ 
compatible with the conformal
structure on $N$ such that
$JP=P$ and $JP'=P'$.   Equivalently, $\dist(x,P')$ is independent of $x$ for $x\in P$ with $|x|=1$.}
\end{equation}

As in \S\ref{bumpy-section}, we let $\MM$ be the space of pairs $(\gamma, [F])$ where $\gamma\in C^\infty(N)$
and where $F$ is a smooth, simple, $e^{\gamma}g_0$-minimal immersion of a
closed $2$-dimensional manifold into $N$.

Let $\Kk$ be the set of $(\gamma, [F])$ in $\MM$ such that $F$ is {\bf not} strongly self-transverse.

If $(\gamma, [F])\in \MM$ and if $F$ is strongly transverse, then $F$ has double points but no
triple points.
Let $\Kk'$ be the set of $(\gamma, [F])$ in $\MM\setminus \Kk$ such that 
\begin{enumerate}
\item $F$ is strongly self-transverse,
\item  $F$ is not an embedding, and 
\item at each self-intersection,
the two tangent planes have the property~\eqref{planes-property}.
\end{enumerate}
Note that $\Kk$ is a closed subset of $\MM$ and that $\Kk'$ is a relatively closed
subset of the open set $\UU:= \MM\setminus \Kk$.

It suffices to show that $\Pi(\Kk\cup \Kk')$ is meager in $C^\infty(N)$.  Since $\Pi(\Kk)$ is meager
in $C^\infty(N)$ (by Theorem~\ref{strong-selfie-theorem}),
 it suffices to show that $\Pi(\Kk')$ is meager in $C^\infty(N)$.
By Corollary~\ref{bumpy-corollary},
 it suffices to show that if $\mathcal{O}$ is an open subset of $\MM_\textnormal{reg}$
that contains a point $(\gamma,[F])$ in $\Kk'$, then
$\mathcal{O}$ also contains a point not in $\Kk'$.

By replacing the background metric $g_0$ by $e^{\gamma}g_0$, we can assume that $\gamma=0$.

By definition of $\Kk'$, there are distinct points
$p$ and $q$ in the domain $M$ of $F$ such that $F(p)=F(q)$ and such that the two tangent planes
to $F(M)$ at $F(p)=F(q)$ belong to $\Pp$.    Since $(\gamma,[F])\notin \Kk$,  the two planes
cross transversely.

Let $f$ be any smooth normal vectorfield on $F$ such that 
\begin{equation}\label{skew}
\begin{gathered}
 f(p)=f(q)=0, \\
 Df(p)=0, \\
 Df(q)\ne 0, \,\text{and} \\
 Df(q)v=0\, \text{for some nonzero vector $v$}.
 \end{gathered}
\end{equation}

Now suppose that $t\mapsto F_t$ is a one-parameter family of immersions with $F_0=F$ and with
$(d/dt)_{t=0}F_t=f$.  By transversality,  there are one-parameter families $p_t$ and $q_t$ with
$p_0=p$ and $q_t=q$ such that $F_t(p_t)=F_t(q_t)$ for $t$ near $0$.
For all sufficiently small $t\ne 0$, \eqref{skew} implies
that the two tangent planes to $F_t(M)$ at $F_t(p_t)=F_t(q_t)$ do {\bf not} have 
the property~\eqref{planes-property}.

It remains only to show that we can choose $f$ and the family $F_t$ so that
$(\gamma_t, [F_t]) \in \MM$ for some smooth $1$-parameter family $\gamma_t\in C^\infty(N)$ 
with $\gamma_0=0$.
For then we will have for all sufficiently small $t\ne 0$ that $(\gamma_t, [F_t])$ is in $\mathcal{O}$ but
not in $\Kk'$.  (It is not in $\Kk'$ because $F_t(M)$ has pair of tangent planes that violate the
property~\eqref{planes-property}.)

As in \S\ref{notation-section},
 we let $W$ be an open subset of $N$ such that $U:=F^{-1}(W)$ contains a point from each component
of $N$ and such that $F|_U$ is an embedding.  (Such a $W$ exists since $F$ is simple.)
By the Runge-Type Theorem~\ref{runge-type-theorem}, there exists a smooth normal vectorfield $f$ to $F$
such that $Jf$ is supported in a compact subset of $U$ and such that~\eqref{skew} holds.
By Theorem~\ref{families-theorem}, there exist $\eps>0$ and smooth one-parameter families
\begin{align*}
   &t\in (-\eps,\eps)\mapsto \gamma_t \in C^\infty(M), \\
   &t\in (-\eps,\eps)\mapsto F_t \in C^\infty(M,N)
\end{align*}
such that $(\gamma_t, [F_t])\in \MM$, $\gamma_0=0$, $F_0=F$, and $(d/dt)_{t=0} F_t = f$.
\end{proof}

\begin{theorem}\label{first-regularity-theorem}
A generic smooth Riemannian metric $g$ on a manifold $N$ has the following property.
If $T$ is a $2$-dimensional, mod $2$ cycle that minimizes $g$-area in its mod $2$ homology class,
or, more generally, that is locally $g$-area-minimizing, then $T$
is a smooth embedded minimal surface with multiplicity $1$.
\end{theorem}

Here ``mod $2$ cycle'' means ``flat chain mod $2$ with boundary $0$''
and ``$T$ is  locally area-minimizing'' means that every point in $N$ has a neighborhood
$\BB$ such that the area of $T\llcorner \BB$ is less than equal to the area
of any mod $2$ flat chain $T'$ in $\BB$ with $\partial T'=\partial T$.

Every $2$-dimensional locally area-minimizing cycle mod $2$ is a smoothly immersed
minimal surface, and if $P$ and $P'$ are two distinct tangent planes at a self-intersection point,
the $P$ and $P'$ are orthogonal: $P$ lies in the orthogonal complement of $P'$.
Thus for mod $2$ cycles, neither the Chang-De-Lellis-Spadaro-Spolaor Theorem nor Moore's Theorem about
generic absence of branch points is needed.
Otherwise, the proof is identical to the proof of Theorem~\ref{first-regularity-theorem}.

\nocite{pedrosa-ritore}
\nocite{hoffman-wei}
\newcommand{\hide}[1]{}

\end{document}